\RequirePackage{fix-cm}

\documentclass[smallextended]{svjour3}
\smartqed  % flush right qed marks, e.g. at end of proof

\usepackage[utf8]{inputenc}
\usepackage{enumerate}
\usepackage{graphicx}
\usepackage{epstopdf}
\usepackage{amsfonts}
\usepackage{amssymb}
\usepackage{amsmath}
\usepackage{wasysym}
\usepackage{hyperref}
\usepackage{mathtools}
\usepackage{bibentry}
\usepackage{comment}
\usepackage{parskip}

\def\dist{\textrm{dist}}

\begin{document}
  
\title{A relaxed version of \v{S}olt\'{e}s's problem and cactus graphs}

\author{Jan Bok \and Nikola Jedli\v{c}kov\'{a} \and Jana Maxov\'{a}}

\institute{Jan Bok \at
Computer Science Institute, Faculty of Mathematics and Physics, Charles University,
Malostransk\'{e} n\'{a}m. 25, 118~00~~Praha~1, Prague, Czech Republic. \\
ORCID: 0000-0002-7973-1361\\
\email{bok@iuuk.mff.cuni.cz}\\
This author is the corresponding author.
\and
Nikola Jedli\v{c}kov\'{a} \at
Department of Applied Mathematics, Faculty of Mathematics and Physics, Charles University,
Malostransk\'{e} n\'{a}m. 25, 118~00~~Praha~1, Prague, Czech Republic. \\
ORCID: 0000-0001-9518-6386\\
\email{jedlickova@kam.mff.cuni.cz}
\and
Jana Maxov\'{a} \at
Department of Mathematics, Faculty of Chemical Engineering, University of Chemistry and Technology,
Technick\'{a} 5, 166~28~Praha~6, Prague, Czech Republic. \\
\email{maxovaj@vscht.cz}
}

\date{Received: date / Accepted: date}

\maketitle

\begin{abstract}
The \emph{Wiener index} is one of the most widely studied parameters in chemical graph
theory. It is defined as the sum of the lengths of the shortest paths between
all unordered pairs of vertices in a given graph. In 1991, Šoltés posed the
following problem regarding the Wiener index: Find all graphs such that its Wiener
index is preserved upon removal of any vertex. The problem is far from being
solved and to this day, only one graph with such property
is known: the cycle graph on 11 vertices.

In this paper, we solve a relaxed version of the problem, proposed by Knor et
al.\ in 2018. For a given $k$, the problem is to find (infinitely
many) graphs having exactly $k$ vertices such that the Wiener index remains the same 
after removing any of them. We call these vertices
\emph{good} vertices and we show that there are infinitely many cactus graphs
with exactly $k$ cycles of length at least 7 that contain exactly $2k$ good
vertices and infinitely many cactus graphs with exactly $k$ cycles of length
$c \in \{5,6\}$ that contain exactly $k$ good vertices. On the other hand, we
prove that $G$ has no good vertex if the length of the longest cycle in $G$ is
at most $4$.

\end{abstract}

\section*{Note on the conference proceedings version of this paper}

An extended abstract of this paper has appeared as \emph{On relaxed
\v{S}olt\'{e}s's problem} in the proceedings of EUROCOMB 2019 conference. The
proceedings were published in Acta Mathematica Universitatis Comenianae.
Compared to the proceedings, this paper has full proofs and it is
self-contained. Furthermore, the experimental results were added.

\section{Introduction}
\label{s:uvod}

The \emph{Wiener index} (also \emph{Wiener number}) is a topological index of
a connected graph, defined as the sum of the lengths of the shortest paths
between all unordered pairs of vertices in the graph. In other words, for a
connected graph $G=(V,E)$, the Wiener index $W(G)$ is defined as $$W(G) :=
\sum_{\{u,v\} \subseteq V}\mathrm{dist}_G(u,v).$$ The index was originally
introduced in 1947 by Wiener~\cite{wiener}  for the purpose of determining the
approximation formula of the boiling point of paraffin. Since then, Wiener
index has become one of the most frequently used topological indices in
chemistry, since molecules are usually modelled by undirected graphs. The
definition of Wiener index in terms of distances between vertices of a graph
was first given by Hosoya~\cite{hosoya}. The index was subsequently
extensively studied by many mathematicians. Apart from pure mathematics, many
applications of Wiener index were found in chemistry, cryptography, theory of
communication, topological networks and others. The quantity is used in
sociometry and the theory of social networks, since it provides a robust
measure of network topology~\cite{estrada}. We refer the interested reader to
the numerous surveys of both applications and theoretical results regarding
the Wiener index, e.g.\
\cite{bonchev2002wiener,dobrynin2001wiener,knor2014wiener,knor2016,Xu2014461}.
We would like to point out that there is also a rich literature about Wiener
index of unicyclic graphs, e.g.\
\cite{dong2012maximum,hou2012maximum,liu2011wiener}. Also, Section 4 of the
aforementioned paper of Furtula et al.\ \cite{Xu2014461} is completely devoted
to unicyclic graphs.

An interesting question regarding the Wiener index is to study how small changes in
a graph affect its Wiener index. On the one hand, it is clear that with edge
removal, the Wiener index increases. On the other hand, the effect of deleting
a vertex is not so clear. Šoltés studied such changes in graphs~\cite{soltes}
and he noticed that the equality $W(C_{11})=W(C_{11} - v) = W(P_{10})$ holds for
every vertex $v \in V(C_{11})$.

Connected graphs that satisfy the equality $W(G) = W(G-v)$ for all  $v \in V(G)$ are
called \emph{Šoltés's graphs}. Šoltés found just one such graph --- the cycle
on eleven vertices $C_{11}$. To this day, this is the only known graph and it
is not known if there exists any other. Thus the following problem, posed by
Šoltés~\cite{soltes} in 1991, is still open.

\begin{problem}[Šoltés's problem]
Find all graphs $G$ such that the equality $W(G) = W(G-v)$ holds for every $v \in V(G)$.
\end{problem}

We remark the existence of graphs satisfying inequality $W(G) \ge W(G-v)$ for
every vertex $v$ of $G$, e.g.\ complete graphs, and there are graphs satisfying inequality
$W(G) \le W(G-v)$ for every vertex $v$ of $G$, e.g. \ cycles up to 10 vertices.

From now on, we assume that all graphs are connected unless we say otherwise.
Otherwise, the problem would become trivial, since the Wiener index of disconnected graphs
is defined as infinity and thus every disconnected graph would be a Šoltés graph.

Knor, Majstorovi\'{c} and \v{S}krekovski~\cite{knor_druhy} defined and studied the following relaxed version of Šoltés's problem.

\begin{problem}~\cite{knor_druhy}
Find all graphs $G$ in which the equality $W(G) = W(G-v)$ holds for at least one vertex $v \in V(G)$.
\end{problem}

A vertex $v \in V(G)$ is called a \emph{good vertex} if $W(G) = W(G-v)$ holds.  In this
terminology, a graph is a Šoltés's graph if all its vertices are good. It was
shown in~\cite{knor2018} that there exist infinitely many unicyclic graphs
with at least one good vertex of degree 2. In~\cite{knor_druhy}, the same
authors found for a given $k \geq 3$ infinitely many graphs that have a good
vertex of degree $k$ and infinitely many graphs with a good vertex of degree
$n-2$ and $n-1$. Furthermore, they proved that dense graphs cannot be Šoltés's
graphs. They also posed the following problem in~\cite{knor_druhy}.

\begin{problem}
\label{p:k good}
For a given $k$, find infinitely many graphs $G$ for which the equality $$W(G) = W(G-v_1) = W(G-v_2)= \ldots = W(G-v_k)$$ holds for distinct vertices $v_1, \dots v_k \in V(G)$.
\end{problem}

In this paper, we solve this problem by finding such an infinite class of graphs
within the class of \emph{cacti}. We recall that \emph{cactus} is a graph
where every edge belongs to at most one cycle. Let us summarize our main
results.

\begin{itemize}
\item We found infinitely many cactus graphs with exactly $k$ cycles of length at least 7 that contain exactly $2k$ good vertices (Theorem \ref{t:exactly2k})
and infinitely many cactus graphs with exactly $k$ cycles of length $c \in \{5,6\}$ that contain exactly $k$ good vertices (Theorem \ref{t:exactlyk}).
\item We prove that $G$ has no good vertex if the length of the longest cycle in $G$ is at most $4$ (Theorem \ref{t:girth4}).
\end{itemize}

\section{Preliminaries}
\label{s:prelim}

All graphs in this paper are simple, undirected and connected. As our results refine and
extend those  of~\cite{knor2018} and~\cite{knor_druhy}, most of the time, we
follow the notation introduced there.

Let $G$ be a connected graph and let $v$ be a vertex in $V(G)$. By $d_G(v)$ we
denote the degree of $v$ in $G$. A \emph{pendant vertex} is a vertex of degree
one and a \emph{pendant edge} is the only edge incident to a pendant vertex.
Note that Wiener index can also be written as 
$$W(G) = \frac 12 \sum _{v \in V(G)} t_G(v),$$ where $t_G(v)$, the \emph{transmission} of $v$ in $G$, is the
sum of distances between $v$ and all the other vertices of $G$.

The complete graph $K_n$ has the smallest Wiener index among all graphs on $n$
vertices since the distance between any two distinct vertices is at least one
in any graph. It is well known that for any connected graph on $n$ vertices,
the maximum Wiener index is obtained for the path $P_n$. Thus, for every graph
$G$ on $n$ vertices, we have
$$\binom{n}{2} = W(K_n) \leq W(G) \leq W(P_n) = \binom{n+1}{3}.$$

It is easy to see that for the Wiener index of the cycle of length $n$, it holds that 
$$W(C_n)= \left\{ 
\begin{array}{ll}
\frac{n^3}8  & \text{if } n \text{ is even,}\\[1mm]
\frac{n^3-n}8  & \text{if } n \text{ is odd.} 
\end{array}
\right.
$$

The proof of the following proposition is also straightforward.

\begin{proposition} \label{prop:pendant}
Let $G$ be a connected graph. Take a new vertex $z$ and connect it by a pendant edge to a vertex $u \in V(G)$. Denote the resulting graph by $G^+$. Then,
$W(G^+) = W(G) + t_G(u) + |V(G)|$.
\end{proposition}

Recall that $v$ is a \emph{good} vertex in $G$ if $W(G) = W(G-v)$.
Let $v_1$ be a fixed vertex in $G$. For a vertex $x \in V(G)$, $x \neq v_1$ we denote (similarly as in~\cite{knor2018})
$$\delta _G (x) = t_G(x) - t_{G-v_1}(x) \quad \text{and} \quad \Delta(G) = W(G) - W(G-v_1).$$

Observe that $\Delta(G)=0$ means that $v_1$ is a good vertex in $G$ and
$\delta_G (x)$ gives us the contribution of the vertex $x$ to $\Delta(G)$.

\section{Infinite families}

First, we need to state a few simple lemmas. We will need them for the proof of Theorem~\ref{t:main}.

\begin{lemma} \label{l:pendant}
Let $G$ be a connected graph with a fixed vertex $v_1$. Take a new vertex $z$ and connect it by a pendant edge to a vertex $u \in V(G)$, $u \neq v_1$. Denote the resulting graph by $G^+$. Hence, $\delta_{G^+}(z) = \delta_{G^+}(u)+1 = \delta_G(u)+1$.
\end{lemma}

\begin{proof}
It is clear that $t_{G^+}(u) = t_{G}(u) +1$ and $t_{G^+-v_1}(u) = t_{G-v_1}(u) +1$. Further it is easy to see that $t_{G^+}(z) = t_{G}(u) +n(G)$ and that $t_{G^+-v_1}(z) = t_{G-v_1}(u) +n(G-v_1)$. Thus, the statement follows.
\qed \end{proof}

\begin{lemma} \label{l:c=7 a vic}
Let $G$ be a connected graph with a fixed vertex $w$. Take a cycle $C_c$ of
length $c \geq 7$ and connect it to $G$ by identifying one vertex on the cycle
with $w$ and denote the resulting graph $G^*$. Let $v_1$  be a neighbor of $w$
on the cycle $C_c$. Hence, 
$$\delta_{G^*}(w) = t_{G^*}(w) - t_{G^*-v_1}(w)\leq -2.$$
\end{lemma}

\begin{proof}
 As $w$ is a cut vertex in  $G^*$, the transmission $t_{G^*}(w)$ is equal to the
 sum of its transmission in $G$ and $C_c$. Thus, $t_{G^*}(w) =
 t_{G}(w)+t_{C_c}(w)$. Similarly, for $G^*-v_1$ we get $t_{G^*-v_1}(w) = t_G(w)+t_{P_{c-1}}(w)$,
 where $w$ is an end vertex of $P_{c-1}$.

 %The same holds for $G^*-v_1$. After deleting $v_1$ from
 %$G^*$ we get $G$ to which we attach a path of length $c-1$ by identifying one
 %of its endpoints with $w$.

We now distinguish two cases according to the parity of $c$.

If $c=2a$ and $a \geq 4$, then
$$\delta_{G^*}(w)= t_{C_{2a}}(w)-t_{P_{2a-1}}(w) = a^2 -(2a^2-3a+1) \leq -2.$$

If $c=2a+1$ and $a \geq 3$, then
$$\delta_{G^*}(w)= t_{C_{2a+1}}(w)-t_{P_{2a}}(w) = a^2+a -(2a^2-a) \leq -2.$$
This completes the proof. \qed
\end{proof}

\begin{lemma} \label{l:c=5,6}
Let $G$ be a connected graph with a fixed vertex $w$. Take a cycle $C_c$ of
length $c\in\{5,6\}$ and connect it to $G$ by identifying one vertex on the
cycle with $w$. Let $v_2$  be a vertex in distance $2$ from $w$ on the cycle
$C_c$ and let $v_1$ be the only common neighbour of $w$ and $v_2$ on the cycle
$C_c$. Add a path of length $2$ to $C_c$ by identifying one of its endpoints
with $v_2$ and denote the resulting graph $G^*$.  We conclude that $\delta_{G^*}(w) \leq
-2$.
\end{lemma}

\begin{proof}
By the same argumentation as in the previous lemma, we get that
$\delta_{G^*}(w)= t_{G^*}(w)-t_{G^*-v_1}(w)=-2$ for $c=5$ and $\delta_{G^*}(w)= -5$ for $c=6$.
\qed \end{proof}

The following theorem is the main step towards proving the main result of this paper.

\begin{theorem}\label{t:main}
Let $c,k$ be natural numbers.
\begin{itemize}
  \item If $c\in\{5,6\}$, then there exist infinitely many cactus graphs with exactly $k$ cycles of length $c$ that contain at least
 $k$ good vertices.
  \item  If $c\geq 7$, then there exist infinitely many cactus graphs with exactly $k$ cycles of length $c$ that contain at least $2k$ good vertices.
\end{itemize}
\end{theorem}

\begin{proof}
Our construction uses similar techniques as in~\cite{knor2018}. We proceed in four steps by constructing graphs $G_1$, $G_2$, $G_3$ and $G_4$. The choice of $G_1$ is different for $c\in\{5,6\}$ and for $c\geq 7$. Therefore, we distinguish two cases.

\begin{case}[$c\in\{5,6\}$]
Let $H$ be a cycle $C_c$ with a path of length $2$ attached to it by identifying one of its endpoints with a vertex on the cycle. Take $k$ copies of $H$ and denote them $H_1, \ldots, H_k$. Fix a vertex $v_0^i \in V(H_i)$ in distance two from the only vertex of degree $3$ in $H_i$. Join $H_1, \ldots, H_k$ together by identifying all $v_0^i$ and denote this new vertex by $w$. Denote the resulting graph $G_1$ and denote by $v_1 \in V(H_1)$ the only common neighbour of $w$ and the vertex of degree 3 on $C_c$. 
\end{case}
\begin{case}[$c\geq 7$]
Take $k$ copies of a cycle $C_c$, fix a vertex in each copy and identify all fixed vertices to one vertex $w$. Denote the resulting graph $G_1$ and denote by $v_1$ any neighbor of $w$ in $G_1$.
\end{case}

For both cases, the graphs $G_1$ are depicted in Fig.~\ref{fig:G1}.

\begin{figure}[h]
    \centering
    \includegraphics[scale=0.64]{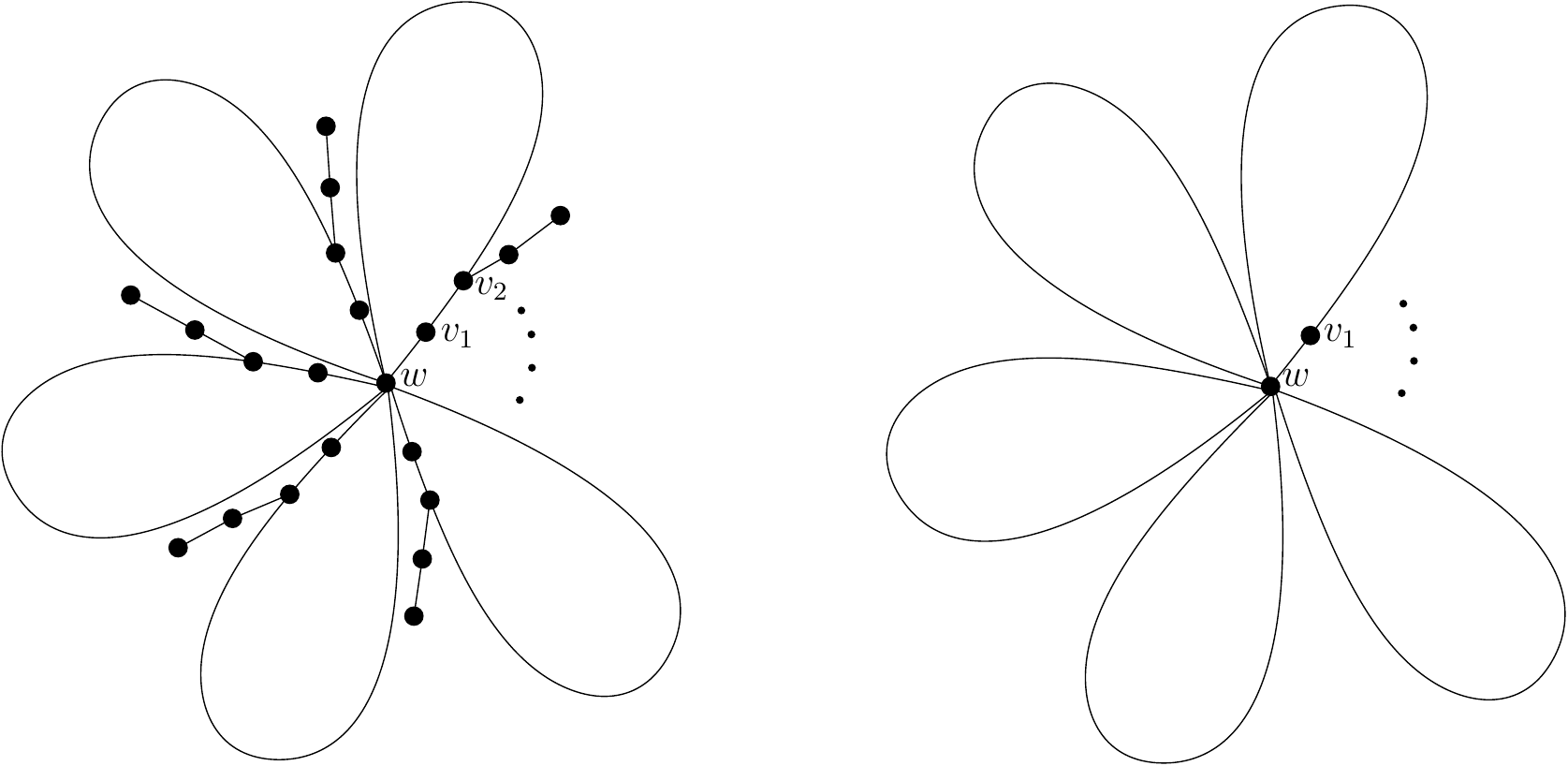}
    \caption{Graphs $G_1$ for $c\in\{5,6\}$ (on the left) and for $c \geq 7$ (on the right).}
    \label{fig:G1}
\end{figure}

Note that in both cases $\delta_{G_1}(w) \leq -2$, by Lemmas~\ref{l:c=7 a vic} and~\ref{l:c=5,6}.

Set $d :=-\delta_{G_1}(w)$ and let $P^d \coloneqq u_d u_{d-1} \ldots u_1 u_0$ be a path of length $d$. Note that $d\geq 2$. Attach $P^d$ to $w$  by identifying $u_d$ with $w$ and denote the resulting graph $G_2$. The crucial observation follows immediately by iterative
use of Lemma~\ref{l:pendant}, namely $\delta_{G_2}(u_i) =-i$. In other words, the value of $\delta_{G_2}(u_i)$ increases along the path $P^d$ from $\delta_{G_2}(w) =-d$ to $\delta_{G_2}(u_0) =0$.

\begin{itemize}
  \item If $\Delta(G_2)=0$, we set $G_3:=G_2$.
  \item If $\Delta(G_2)<0$, we connect exactly $-\Delta(G_2)$ new pendant vertices to $u_0$ in $G_2$ and  denote the resulting graph $G_3$. As $\delta_{G_2}(u_0) =0$, by Lemma~\ref{l:pendant} the contribution $\delta_{G_3}(x)$ of any pendant vertex $x \in V(G_3) \setminus V(G_2)$ to $\Delta(G_3)$ is $\delta_{G_3}(x)=1$ and thus  $\Delta(G_3)=0$.
  \item If $\Delta(G_2)>0$, we connect exactly $\Delta(G_2)$ new pendant vertices to $u_2$ in $G_2$ and denote the resulting graph $G_3$. As $\delta_{G_2}(u_2) =-2$, by Lemma~\ref{l:pendant} the contribution $\delta_{G_3}(x)$ of any pendant vertex $x \in V(G_3) \setminus V(G_2)$ to $\Delta(G_3)$ is $\delta_{G_3}(x)=-1$ and thus again $\Delta(G_3)=0$.
\end{itemize}

Finally, for arbitrary $p\geq 0$ we add to $G_3$ exactly $p$ new pendant vertices, connect them all to $u_1$ and denote the resulting graph $G_4$. As $\delta_{G_3}(u_1) =-1$, by Lemma~\ref{l:pendant} we get that for every $x \in V(G_4) \setminus V(G_3)$, the contribution $\delta_{G_4}(x)=0$ and thus $\Delta(G_4)=\Delta(G_3)=0$. In other words, $v_1$ is a good vertex in $G_4$ for any choice of $p$.

\begin{figure}[h]
    \centering
    \includegraphics[scale=0.6]{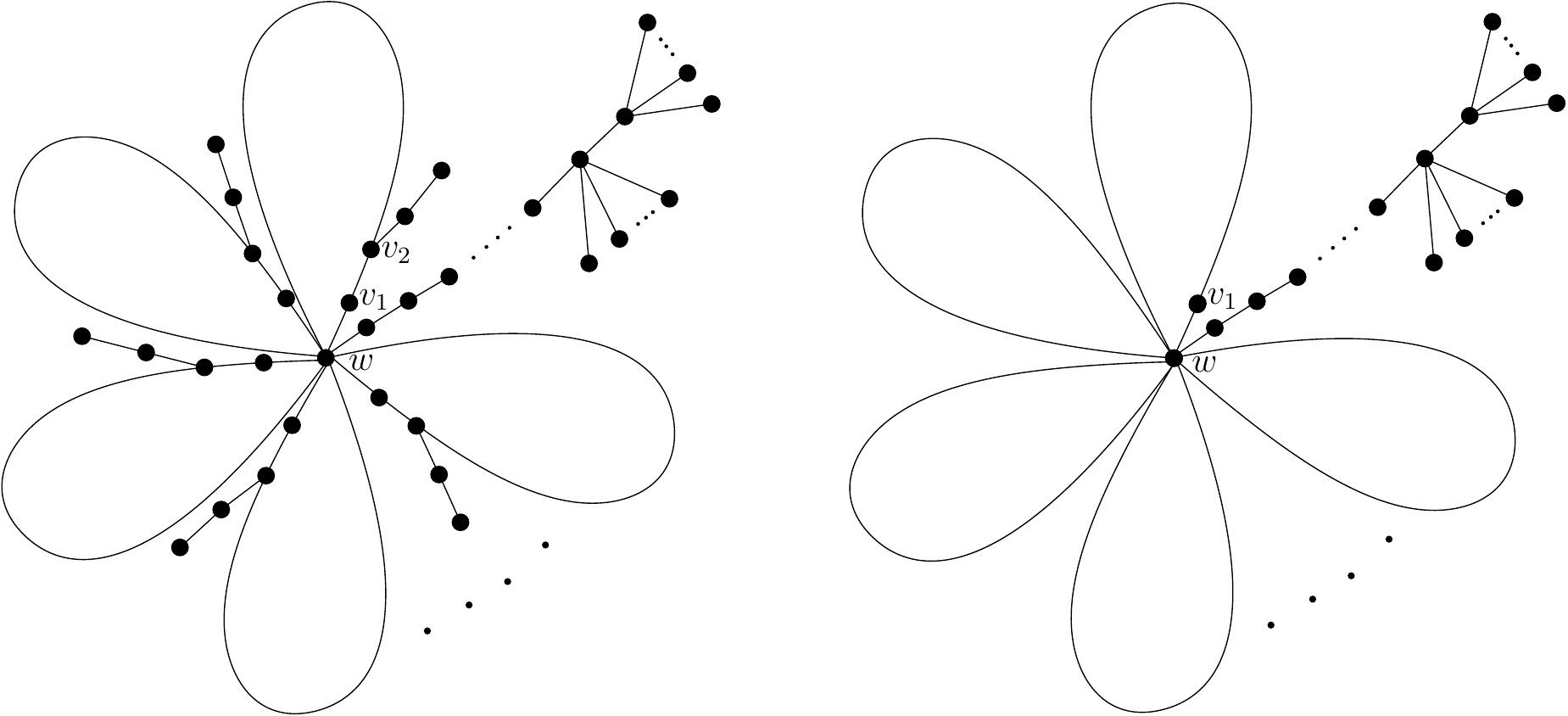}
    \caption{An illustration of construction from Theorem \ref{t:main}. Graphs $G_4$ for $c\in\{5,6\}$ are on the left and for $c \geq 7$ on the right.}
    \label{fig:G4}
\end{figure}

It remains to show that $v_1$ is not the only good vertex in $G$. This follows immediately from the symmetry of the starting graph $G_1$. It is obvious that for $c\in\{5,6\}$, there are other $k-1$ good vertices other than $v_1$ (one in each copy of $C_c$) since we can find one vertex in each cycle such that its removal yields a graph isomorphic to $G_4-v_1$. If $c \geq 7$, we can argue similarly that in $G_4$, there are $2k$ good vertices (two in each copy of $C_c$). This is illustrated in Fig.~\ref{fig:G4}.
\qed \end{proof}

So far, we proved that for every natural $k$, there are infinitely many graphs
with at least $k$ (or $2k$) good vertices. Now, we can state the main result of
the paper which says that the graphs constructed in Theorem~\ref{t:main}
contain no other good vertices.

\begin{theorem}
\label{t:exactly2k}
Let $k$ be a natural number. For every $c \geq 7$, there are infinitely many
cacti graphs with exactly $k$ cycles of length $c$ that contain exactly $2k$
good vertices.
\end{theorem}

\begin{proof}
Consider the graph $G_4$ for $c \geq 7$ from the proof of Theorem \ref{t:main}
and let us denote it by $G$. It follows from Theorem \ref{t:main} that there
are at least $2k$ good vertices in $G$. Now, we prove that there are exactly
$2k$ good vertices in $G$.

As $G-v$ has to be connected, the only good vertices may occur on the cycles.
Pick one of the cycles in $G$ and denote its vertex set by $L$. Recall that $w$ is the only
common vertex of all cycles. We denote the vertices of $L$ by $w=v_0 \ldots
v_{c-1}$, consecutively. Set $K = V(G) \setminus L$.

We know from Theorem \ref{t:main} that $W(G)-W(G-v_1)=W(G)-W(G-v_{c-1})=0$. 
Suppose for a contradiction that $v_i$ is a good vertex for some $i=2, \dots
c-2$. Our goal is to prove that $W(G)- W(G-v_i)>0$ for every $i \in \{ 2,
\dots \lfloor \frac c2 \rfloor \}$. This will complete the proof
of Theorem \ref{t:exactly2k} as for the other vertices in $L$, we can argue
in the same way due to symmetry. 

Let us denote $\Delta' = W(G)- W(G-v_i)$. As $W(G)-W(G-v_1)=0$, we have 
$$\Delta' = W(G)- W(G-v_i)= W(G-v_1) -W(G-v_i).$$
Because of  $V(G)= K \ \dot \cup \ L$, we may rewrite $W(G-v_1)$  and $W(G-v_i)$ as
\begin{align*}
W(G-v_1)=  &\sum_{u,v \in K}\dist_{G-v_1}(u,v) + \sum_{u \in K, v \in L - v_1}\dist_{G-v_1}(u,v) \\ &+
      \sum_{u,v \in L - v_1}\dist_{G-v_1}(u,v),
\end{align*}
and analogously,     
\begin{align*}
W(G-v_i) = &\sum_{u,v \in K}\dist_{G-v_i}(u,v) + \sum_{u \in K, v \in L - v_i}\dist_{G-v_i}(u,v) \\ &+\sum_{u,v \in L - v_i}\dist_{G-v_i}(u,v).
\end{align*}
It is obvious that the equation $\dist_{G-v_1}(u,v)=\dist_{G-v_i}(u,v)$ holds for all $u,v \in K $.
Further it is clear that 
$$\sum_{ u,v \in L - v_1} \dist_{G-v_1}(u,v)=\sum_{ u,v \in L - v_i} \dist_{G-v_i}(u,v) = W(P_{c-1}).$$

Using these observations, we obtain that
\begin{align*}
\Delta' &= W(G-v_1) - W(G-v_i) \\
&= \sum_{ u \in K,v \in L - v_1} \dist_{G-v_1}(u,v) - \sum_{ u \in K,v \in L - v_i} \dist_{G-v_i}(u,v).
\end{align*}

%\begin{align*}
%\Delta' &= W(G-v_1) - W(G-v_i) \\
%&= \sum_{ u \in K,v \in L - v_1} \dist_{G-v_1}(u,v) - \sum_{ u \in K,v \in L - v_i} \dist_{G-v_i}(u,v).  
%\end{align*}

Let us define $L':=  L - \{ v_1,v_i \}$. We can also write these sums as 
\begin{align*}
\Delta'
= &\sum_{ u \in K, v \in L'} \dist_{G-v_1}(u,v) + \sum_{ u \in K} \dist_{G-v_1}(u,v_i) \\
&\quad -  \sum_{ u \in K, v \in L'} \dist_{G-v_i}(u,v) - \sum_{ u \in K} \dist_{G-v_i}(u,v_1)
%&= \sum_{ u \in K} (\dist_{G-v_1}(u,v_i) - \dist_{G-v_i}(u,v_1))\\ 
%&\quad + \sum_{ u \in K,v \in L'} (\dist_{G-v_1}(u,v)-\dist_{G-v_i}(u,v)).
\end{align*}

%Note that it holds $\dist_{G-v_1}(u,v) = \dist_{G-v_1}(u,w) +\dist_{G-v_1}(w,v) $ for all $u \in K, v \in L'$. That is
% because $w$ is a cut-vertex in $G$. Analogously, $\dist_{G-v_i}(u,v) = \dist_{G-v_i}(u,w) + \dist_{G-v_i}(w,v).$ 

As $w$ is a cut vertex in $G$, $\dist_{G-v_1}(u,v) = \dist_{G-v_1}(u,w)
+\dist_{G-v_1}(w,v) $ holds for all $u \in K, v \in L \setminus \{ v_1 \}$.
Analogously, $\dist_{G-v_i}(u,v) = \dist_{G-v_i}(u,w) +\dist_{G-v_i}(w,v)$
holds for all $u \in K$ and $v \in L \setminus \{ v_i \}$. $u \in K, v \in L
\setminus \{ v_i \}$. Further note that $\dist_{G-v_1}(u,w)=
\dist_{G-v_i}(u,w)$ holds for all $u \in K$.

Thus we get 
\begin{align*}
\Delta' =&\, |K|\cdot \big(\dist_{G-v_1}(v_i,w) - \dist_{G-v_i}(v_1,w)\big) \\
 &+ \sum_{ u \in K,v \in L'} \big(\dist_{G-v_1}(w,v)-\dist_{G-v_i}(w,v)\big).
\end{align*}
It easy to see that
$$\dist_{G-v_1}(v_i,w) - \dist_{G-v_i}(v_1,w) = (c-i-1) > 0.$$
Moreover, for every $v \in L'$, it holds that $\dist_{G-v_1}(w,v) \geq \dist_{G-v_i}(w,v)$. We finally get 
$$\Delta'= |K|\cdot (c-i-1) 
 + \sum_{ u \in K,v \in L'} (\dist_{G-v_1}(w,v)-\dist_{G-v_i}(w,v)) > 0.$$
This completes the proof.
\qed
\end{proof}

\begin{theorem}
\label{t:exactlyk}
Let $k$ be a natural number. For every $c \in \{5,6\}$, there are infinitely many cacti graphs with exactly $k$ cycles of length $c$ that contain exactly $k$ good vertices.
\end{theorem}

\begin{proof}
We will proceed similarly as in the proof of the previous theorem. Again note that good vertices can be located only on cycles.
Let $G:=G_4$ for $c \in \{5,6\}$ from the proof of Theorem \ref{t:main}. We already know that there are at least $k$ good vertices in $G$. It remains to prove that there are exactly $k$ good vertices in $G$.

Again, take one of the cycles in $G$ together with the path of length $2$ attached to it and denote its set of vertices by $L$. Recall that $w$ is the only common vertex of all cycles. Denote the vertices of the cycle in $L$ consecutively by $w=v_0 \ldots v_{c-1}$ such that $v_2$ is the only vertex of degree $3$ in $L$. 
Let $K$ be the set of vertices of $V(G) \setminus L$.

We assume for a contradiction that there is another good vertex except for $v_1$. As $G-v_2$ is not connected, $v_2$ cannot be a good vertex. Our goal is to prove $W(G)-W(G-v_i)>0$ for every $i=3, \dots c-1$.

It is obvious that for all $u,v \in K$, it holds that $\dist_{G-v_1}(u,v)=\dist_{G-v_i}(u,v)$.  
Further it is clear that 
$$\sum_{ \{u,v\} \subset L - v_1} \dist_{G-v_1}(u,v)=W(P_{c+1}).$$
It is well known that for trees on $n$ vertices, the maximum Wiener index is attained exactly for the path $P_n$. Hence, for every $i \in \{3,\ldots,c-1\}$, we have
$$\sum_{ \{u,v\} \subset L - v_i} \dist_{G-v_i}(u,v) \leq \sum_{ \{u,v\} \subset L - v_1} \dist_{G-v_1}(u,v) = W(P_{c+1}).$$

Let us define $\Delta' := W(G)- W(G-v_i)$ and $L':=  L - \{ v_1,v_i \}$. 
By summarizing these observations and the fact that $W(G)-W(G-v_1)=0$, we obtain that
\begin{align*}
\Delta' &= W(G-v_1) -W(G-v_i)\\
&\geq \sum_{ u \in K,v \in L - v_1} \dist_{G-v_1}(u,v) - \sum_{ u \in K,v \in L - v_i} \dist_{G-v_i}(u,v).
\end{align*}
By the same computation as in the previous proof, we further get 
\begin{align*}
\Delta' = &|K|\cdot (\dist_{G-v_1}(v_i,w) - \dist_{G-v_i}(v_1,w)) \\
&+ \sum_{ u \in K,v \in L'} (\dist_{G-v_1}(w,v)-\dist_{G-v_i}(w,v)).
\end{align*}

Note that $\dist_{G-v_1}(v_i,w) \geq \dist_{G-v_i}(v_1,w)$
with equality if and only if $i=c-1$. Also note that 
$\dist_{G-v_1}(w,v) \geq \dist_{G-v_i}(w,v)$ for all $v \in L'$.

Let $v$ be the only pendant vertex in $L$, i.e. the end of the attached path of length $2$. For the vertex $v$, we have
$$\dist_{G-v_1}(w,v) - \dist_{G-v_i}(w,v) = c-4 >0.$$
We conclude that for every $i=3, \dots c-1$, $W(G)-W(G-v_i)>0$. Thus none of the vertices $v_i$, $i=3, \dots c-1$ can be a good vertex in $G$, a contradiction.
\qed
\end{proof}

Let us remark that if we define $G_1$ for $c \geq 7$ as it was done for $c \in
\{5,6\}$ (that is we add a path of length $2$ to the vertex $v_2$), we
obtain graphs with exactly $k$ good vertices also for $c\geq 7$ and for arbitrary
$k$.

Furthermore, for $k=1$, we obtain precisely the graphs constructed in
\cite{knor2018}. It follows from our results that their unicyclic graphs have
exactly one good vertex if the length $c$ of the unique cycle is 5 or 6 and
exactly two good vertices in the case when $c \geq 7$. Let us note that this fact
was not proved in~\cite{knor2018} and only the existence of at least one good
vertex was shown there.

\section{Negative results}

The following theorem explains why we cannot hope for a similar result when the cycle length $c$ is equal to $3$ or $4$.

\begin{theorem}\label{t:girth4}
Let $G$ be a connected graph which is not a tree. If the length of the longest cycle in $G$ is at most 4, then $G$ has no good vertex.
\end{theorem}
\begin{proof}
Suppose for a contradiction that $G$ has a vertex $v$ such that $W(G) = W(G-v)$. It is clear that according to Proposition~\ref{prop:pendant}, $v$ cannot be a pendant vertex. It is obvious that $v$ has to lie on a cycle, otherwise $G-v$ would be a disconnected graph. Note that by deleting $v$ from $G$, the distance between each pair of vertices in $G-v$ remains the same as in $G$. It follows that $W(G-v)= W(G) - t_G(v)$ and hence $W(G-v)<W(G)$, a contradiction.
\qed
\end{proof}

\section{Experimental results}
\label{s:experimental}

During our research we also made a few computer-run experiments to make a census of unicyclic
graphs given prescribed number of vertices and number of good vertices.
Let us define a function $g_k(\mathcal{G})$ for a class of graphs $\mathcal{G}$ in the following
way:
$$g(\mathcal G,k) :=  | \{ G \in \mathcal G : \textrm{number of good vertices in } G \textrm{ is } k \} |.$$
Table
\ref{tab:unicyclic} sums our experiments for $\mathcal G = \mathcal U_n$, that is non-isomorphic connected unicyclic graphs
on $n$ vertices. We would like to point out several observations.
\begin{itemize}
  \item Up to order 8, there is no graph with a good vertex.
  \item The cycle $C_{11}$ is the only Šoltés's graph among all connected unicyclic graph up to order 18.
  \item There is one graph ($G_{12}$, depicted in Fig.~\ref{fig:6dobrych}) that has 6 good vertices. Note that $6=\frac{1}{2}|V(G_{12})|$ and see also Problem \ref{prob:atleasthalf}.
  \item All other unicyclic graphs up to 18 vertices have at most 4 good vertices.
  \item It does not hold that for every two good vertices $u,v \in V(G)$, there is
  an automorphism $f$ such that $f(u)=v$. This is illustrated on two graphs in Fig.~\ref{fig:auto}.
\end{itemize}

\begin{table}
\centering
\setlength{\tabcolsep}{4pt}
\begin{tabular}{|p{1.1cm}|r|r|r|r|r|r|r|r|r|r|}
\hline
$n$                     & 9   & 10  & 11  & 12  & 13  & 14  & 15  & 16  & 17  & 18  \\ \hline \hline
$|\mathcal U_n|$   & 240   & 657   & 1806   & 5026   & 13999   & 39260   & 110381   & 311465   & 880840   & 2497405   \\ \hline
$g(\mathcal U_n,1)$   & 1   & 1   & 3     & 21  & 62  & 207 & 599   & 1747  & 5040  & 13838 \\ \hline
$g(\mathcal U_n,2)$   & 0   & 1   & 3   & 9   & 16  & 34  & 90  & 229   & 483   & 1303  \\ \hline
$g(\mathcal U_n,3)$   & 0   & 0   & 0   & 0   & 0   & 0   & 1   & 5   & 0   & 30  \\ \hline
$g(\mathcal U_n,4)$   & 0   & 0   & 0   & 1   & 0   & 1   & 2   & 7   & 0   & 22  \\ \hline
$g(\mathcal U_n,6)$   & 0   & 0   & 0   & 1   & 0   & 0   & 0   & 0   & 0   & 0   \\ \hline
$g(\mathcal U_n,11)$   & 0   & 0   & 1   & 0   & 0   & 0   & 0   & 0   & 0   & 0   \\ \hline
\end{tabular}
\vspace{0.2cm}
\caption{Values of $g_k(\mathcal U_n)$ for various $k$ and $n$. For all $k > 6$ and listed $n$'s, the values are zero.}
\label{tab:unicyclic}
\end{table}

\begin{figure}[h]
    \centering
    \includegraphics[scale=1.4]{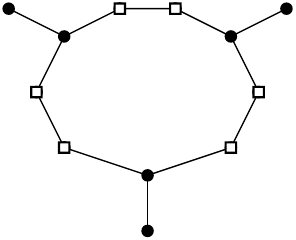}
    \caption{A graph of order 12, $G_{12}$, with exactly 6 good vertices. Good vertices are depicted as squares.}
    \label{fig:6dobrych}
\end{figure}

\begin{figure}[h]
    \centering
    \includegraphics[scale=0.9]{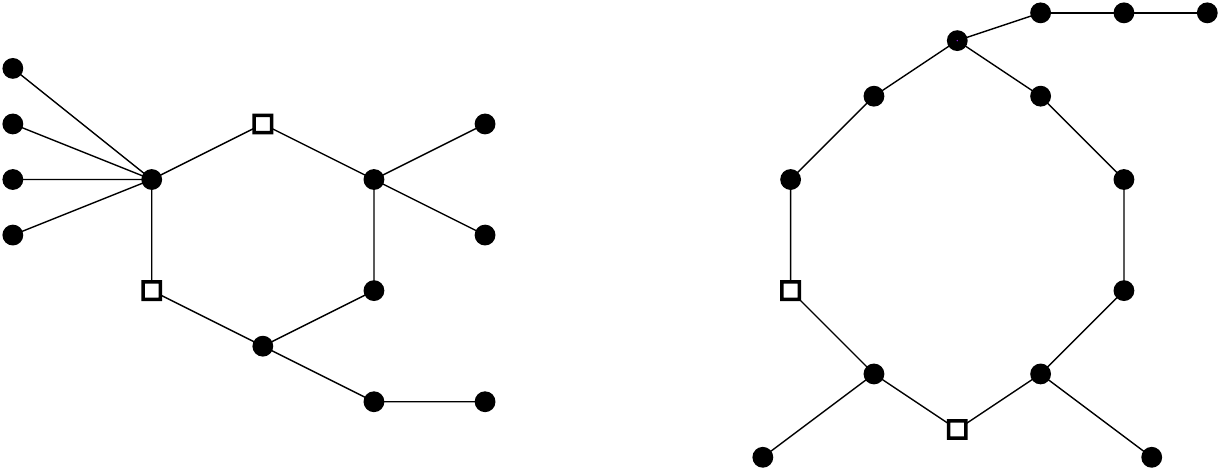}
    \caption{An example of two graphs having two good vertices $u,v$ such that there is no automorphism $f$ of the graph for which $f(u)=v$. Good vertices are depicted as squares.}
    \label{fig:auto}
\end{figure}

\section{Conclusions}

Šoltés's problem is still far from being resolved and its solution is the
ultimate goal. Aside from this problem, the partial results provided in this
paper and also the results of Knor et al.~\cite{knor_druhy,knor2018} are
important on their own, since they show us how does Wiener index change upon
slight modifications of a graph.

Finally, we would like to mention two natural relaxations of Šoltés's problem
that seem to be challenging.

We have seen that in $G_{12}$, half of the vertices are good. So far, $G_{12}$
is the only known graph with this property. It would be interesting to find
more (infinitely many) such graphs.

\begin{problem} \label{prob:atleasthalf}
  Are there infinitely many graphs $G$, such that $G$ has at least $\frac{1}{2}|V(G)|$ good vertices?
\end{problem}

We would like to mention that very recently, Akhmejanova et al.~\cite{new} made a significant effort towards solving Problem~\ref{prob:atleasthalf}.

Another interesting open question is the existence of graphs with only a few
(at most k) ``bad'' vertices.

\begin{problem} \label{prob:fewbad}
For a given $k$, find infinitely many graphs $G$ for which the equality
$$W(G) = W(G-v_1) = W(G-v_2)= \ldots = W(G-v_{n-k})$$
holds for $k$ distinct vertices $v_1, \dots v_{n-k} \in V(G)$.
\end{problem}

\begin{acknowledgements}
The first and the second author would like to acknowledge the support of the
grant SVV-2017-260452. The second author was
supported by Student Faculty Grant of Faculty of Mathematics and Physics,
Charles University.
\end{acknowledgements}

\section*{Declarations}
The authors declare that they have no conflict of interest.

\bibliographystyle{spmpsci}
\bibliography{bibliography}
  
\end{document}